\documentclass[a4paper,11pt]{article}	
\usepackage[italian,english]{babel}	
\usepackage{amsmath,amssymb,amsthm, hyperref}
\usepackage{easyReview}
\usepackage[nameinlink]{cleveref}
\usepackage{amsfonts} 
\usepackage{graphicx} 
\usepackage{braket}
\usepackage{times}	
\usepackage{ifthen}
\usepackage{xspace}
\usepackage{fancyhdr}
\usepackage{mathrsfs}
\usepackage{enumerate}
\usepackage{verbatim}
\usepackage{hyperref}
\usepackage{mathrsfs}
\usepackage[applemac]{inputenc}

\usepackage{fancyhdr}

\usepackage{fancyhdr}

\newtheorem{theorem}{Theorem}[section]
\newtheorem{lemma}[theorem]{Lemma}
\newtheorem{definition}[theorem]{Definition}
\newtheorem{proposition}[theorem]{Proposition}
\newtheorem{remark}[theorem]{Remark}

\newtheorem{example}[theorem]{Example}

\newcommand{\R}{\mathbb{R}}

\newcommand{\cH}{\mathcal{H}}

\newcommand{\N}{\mathbb{N}}
\newcommand{\T}{\mathbb{T}}

\newcommand{\cP}{\mathcal{P}}
\newcommand{\cL}{\mathcal{L}}

\newcommand{\eps}{\varepsilon}
\newcommand{\diver}{{\rm div}}

\DeclareMathOperator{\Wass}{\mathbf{d}_1}

\numberwithin{equation}{section}

\title{A note on first order quasi-stationary Mean Field Games}
\author{Fabio Camilli \thanks{Dip. di Scienze di Base e Applicate per l'Ingegneria,  Universit{\`a}  di Roma  ``La Sapienza", via Scarpa 16, 00161 Roma, Italy  ({\tt  fabio.camilli@uniroma1.it}).} \and Claudio Marchi \thanks{Dip. di Matematica "Tullio Levi-Civita'', Universit\`a di Padova, via Trieste, 35121 Padova, Italy ({\tt marchi@math.unipd.it}).} \and Cristian Mendico \thanks{Institut de Math\'ematique de Bourgogne, UMR 5584 CNRS, Universit\'e Bourgogne, 21000 Dijon, France ({\tt cristian.mendico@u-bourgogne.fr})}}

\begin{document}
\maketitle
\begin{abstract}
Quasi-stationary Mean Field Games models consider agents who base their strategies on current information without forecasting future states. In this paper we address the first-order quasi-stationary Mean Field Games system, which involves an ergodic Hamilton-Jacobi equation and an evolutive continuity equation. Our approach relies on weak KAM theory. We introduce assumptions on the Hamiltonian and coupling cost to ensure continuity of the Peierls barrier and the Aubry set over time. These assumptions, though restrictive, cover interesting cases such as perturbed mechanical Hamiltonians.
\end{abstract}

\noindent\textit{Keywords:} Mean Field Games, weak KAM theory, viscosity solutions. \\
\noindent\textit{  Mathematics Subject Classification:} 35B40, 35Q49, 49N80, 37J51.
	
	

\section{Introduction}

Mean Field Games (MFG) theory is a mathematical framework for analyzing decision-making processes among large populations of small interacting agents. It was independently introduced by Lasry-Lions \cite{ll} and Huang-Caines-Malham\'e \cite{hcm}. MFG combines game theory and partial differential equations (PDE) to address problems involving many players, where traditional game theory becomes impractical due to complexity. In this model, agents are fully rational, meaning they have perfect knowledge of the system's dynamics and can predict the evolution of the mean field based on their actions and those of others. The game reaches a mean field equilibrium where all strategies and the mean field distribution are consistent, and no agent can unilaterally improve their outcome.
The corresponding PDE system   includes a backward-in-time Hamilton-Jacobi-Bellman   equation for the agents' strategy evolution and a forward-in-time Fokker-Planck  equation for the state distribution dynamics.\par

However, real-world agents often deviate from full rationality due to unpredictability and the need to learn and adapt their strategies over time. In quasi-stationary   MFG model, introduced by Mouzouni \cite{mou} and studied further in \cite{cm}, the generic agent cannot predict the future evolution of the population. Instead, the agent makes strategic decisions based solely on the information available at the current moment. 
This means that each agent observes the current state of  the mean field $m(t)$ and, without attempting to forecast how this might change over time,   optimize the cost functional
\begin{equation*}
	\liminf_{T \to \infty}\frac{1}{T}   \inf_{a}\mathbb{E}_{x,t}\Big[\int_{t}^{T} L(X(s),  a(X(s)))+F(X(s), m(t))\;ds\Big]
\end{equation*}
where
\[dX(s)=a(X(s))ds+\sqrt{2}dW(s), \qquad X(t)=x \]
and $a$ is a feedback control law. The corresponding MFG system is given by
\begin{equation}  \label{MFG_2order}
\left\{
\begin{aligned}
& - \Delta u + H(x,Du(t) )+\alpha(t) =F(x,m(t)) \quad &\textrm{ in }  \T^d,\,\forall t\in [0,T] \\
&\partial_{t}m -  \Delta m -\diver(m  H_{p}(x,Du(t)))=0&\textrm{ in } \T^d \times(0,T)\\
& m (0)=m_{0}& \textrm{ in } \T^d,
\end{aligned}
\right.
\end{equation}
where $\T^d$ is the unit torus in $\R^d$. 
Let us underline an important feature of problem~\eqref{MFG_2order}: the Hamilton-Jacobi-Bellman equation is stationary while the Fokker-Planck equation is evolutive; hence, the standard structure of MFG systems, where the Fokker-Planck equation is the dual of the linearized Hamilton-Jacobi-Bellman equation, is lost. Nevertheless, the time $t$ still affects the first equation in the system through the time-dependent distribution~$m$ of players; in other words, the time~$t$ plays the role of a parameter in the Hamilton-Jacobi-Bellman equation.  The existence of a solution to \eqref{MFG_2order} can be established by utilizing a continuous dependence estimate for the Hamilton-Jacobi-Bellman   equation. This estimate guarantees the required time regularity of the vector field $D_p H(x,Du(t))$ that drives the Fokker-Planck   equation. Uniqueness of the solution does not require monotonicity  of  the
coupling cost and it is   achieved by applying Gronwall's Lemma, instead of the standard  duality argument. \par
The aim of this paper is to address the well-posedness of the  first order quasi-stationary  MFG system 
\begin{equation}\label{MFG_erg} 
\left\{
\begin{aligned}
(i)\quad	&\cH(x,Du(t), m(t))+\alpha(t)=0 \quad &\textrm{ in } \T^d,\,\forall t\in [0,T],\\
(ii)\quad	&\partial_{t}m-\diver(m  D_p\cH(x,Du(t), m(t)))=0 &\textrm{ in } \T^d \times(0,T),\\
(iii)\quad	& m(0)=m_{0} & \textrm{ in } \T^d.
\end{aligned}
\right.
\end{equation}
In contrast   to the second-order problem, several intriguing issues arise in the first order case and  to prove   existence of solutions   is more challenging. Indeed, for a fixed $t$, while the ergodic constant \(\alpha\) in \eqref{MFG_erg}.(i) is uniquely determined, one cannot expect to have  a unique viscosity solution (even up to an additive constant).  Some selection techniques are available in literature, see \cite{aips} and \cite{dfiz}. However, also using these techniques, continuous dependence  results for the Hamilton-Jacobi equation only ensures the continuity of $u(t)$ with respect to time, without providing any information on the regularity of$Du(t)$,  which is  essential for proving existence of a solution to the system \eqref{MFG_erg} using a fixed point argument.\par
 We investigate the quasi-stationary MFG system \eqref{MFG_erg} using an approach based on weak KAM theory (see \cite{Fa}). The structure of the solution set of the ergodic Hamilton-Jacobi equation is deeply connected with the associated dynamical system, particularly with the properties of the Peierls barriers and the Aubry set (see Section \ref{Peierls} for definitions). Since the Hamiltonian depends on the measure $m(t)$, the Aubry set can vary over time $t$ and typically lacks stability with respect to perturbations. To address this issue, we introduce specific assumptions on the Hamiltonian and on the coupling cost in order to achieve continuity of the Peierls barrier and the Aubry set over time. 
 Although our assumptions are somewhat restrictive, they still encompass interesting cases, such as a perturbed mechanical Hamiltonian. By exploiting this structure, we can prove the existence of a solution to system \eqref{MFG_erg} via a  fixed point argument. However, it is worth mentioning that the structural assumptions in force does not guarantee the uniqueness of solutions due to the instability of the Aubry set w.r.t. time. Indeed, uniqueness (as in \cite{mou}) is strongly related to the continuity of the gradient of $u(t, \cdot)$ w.r.t. time, a property connected  to the stability of the Aubry set  which fails even in simple case as the mechanical system (see, for instance, \cite[Example 1.3]{Qinbo}).   
\par
We mention that Weak KAM theory has been exploited to study the long time behavior of MFG
(see \cite{Bardi, cannarsa, cannarsa_1, card,card_mend, Kaizhi, Kaizhi_1}).
Moreover the first-order quasi-stationary MFG model shares similarities with the Hughes model, another classical model used to study agent behavior with partial rationality \cite{amad}. Like the system \eqref{MFG_erg}, the Hughes model comprises a nonlinear conservation law coupled with a stationary eikonal equation. Also in this case, the primary challenge lies in the irregularity of the gradient of the solution to the eikonal equations, which affects the flux in the conservation law. Existence results for the Hughes model are available only for one-dimensional spatial domains \cite{ag}, although the model can be formulated in any spatial dimension.

\vspace{0.2cm}
The paper is organized as follows. In Section \ref{Peierls} we introduce the tools from weak KAM theory tailored to the MFG structure. Section \ref{quasistat_MFG} is devoted to the proof of the main result on the existence of solutions to the quasi-stationary MFG system \eqref{MFG_erg}.

\vspace{0.2cm}
\subsection*{Acknowledgement}
The authors were partially supported by Istituto Nazionale di Alta Matematica, INdAM-GNAMPA project 2024. The third authors were partially supported by the MIUR Excellence Department Project awarded to the Department of Mathematics, University of Rome Tor Vergata, CUP E83C23000330006.


\section{The ergodic Hamilton-Jacobi equation}\label{Peierls}
In this section, we recall definitions and preliminary results from   weak
KAM theory, which will be used later in this paper (see \cite{Fa} for more details).\par
Let $\cP(\T^d)$ be the set of probability measures on   $\T^d$ which is a compact topological space when endowed with the weak$^*$-convergence. Moreover the topology  on  $\cP(\T^d)$ is metrizable by means of the Kantorovich-Rubinstein distance 
\[
\Wass(\mu,\mu')=\sup\big\{\int_{\T^d} f(x)d(\mu-\mu'):\,f:\T^d\to\R\quad\mbox{is $1$-Lipschitz continuous}\big\}.
\]
For a given $m\in \cP(\T^d)$, we consider the ergodic Hamilton-Jacobi equation
\begin{equation}\label{eq:HJerg}
	\cH(x,Du, m)+\alpha=0 \quad \textrm{ in } \T^d,
\end{equation}
where the unknowns are the ergodic constant $\alpha$ and the viscosity solution $u$. We denote by $\cL$ the   Lagrangian associated to $\cH$, i.e.
\begin{equation}\label{Legendre}
\cL(x, v,m) = \sup_{p \in \R^d} \{\langle p, v \rangle - \cH(x, p,m)\}.
\end{equation}
We make following assumptions on the Hamiltonian $\cH$. 
\begin{itemize}
\item[({\bf H})] For each $m\in\cP(\T^d)$,   $\cH(\cdot,\cdot,m) \in C^2(\T^d \times \R^d)$ and, for each $(x, p) \in \T^d \times \R^d$, the map $m \mapsto \cH(x, p, m)$ is continuous w.r.t. $\Wass$ distance. Moreover, there exists a constant $c_H > 0$ (independent of $m$)  such that for any $(x, p) \in \T^d \times \R^d$ we have
\begin{align*}
\cH(x, p,m)  \geq\; & c_H( |p|^2 - 1),\\
D^2_p \cH(x, p,m) \geq\; & 1/c_H,\\
|D_p \cH(x, p,m)| \leq\; & c_H(1+|p|).
\end{align*}
\end{itemize}
In particular, by the compactness of $\T^d\times\cP(\T^d)$, for each $R>0$, there exists a modulus of continuity $\omega_R: [0,\infty) \to [0, \infty)$ such that for any $x \in \T^d$, any $p\in B_R$, and any $m_1, m_2 \in \cP(\T^d)$ we have
\begin{equation}\label{eq:cH_m}
|\cH(x,p, m_1) - \cH(x,p, m_2)| \leq \omega_R(\Wass(m_1, m_2)).
\end{equation}
It is well known that the Lagrangian $\cL$ satisfies the same properties of $\cH$. Given a probability measure $m\in\cP(\T^d)$, we define the family of action functions
\[
A^{m}:  [0,\infty)\times\T^{d}\times \T^{d} \to \R
\]
as 
\begin{equation*}
A^{m}(\tau, x, y) = \inf\left\{\int_{0}^{\tau} \cL(\gamma(s), \dot\gamma(s), m )\;ds : \gamma \in AC([0,\tau]; \T^d), \; \gamma(0)=x,\; \gamma(\tau)=y \right\},
\end{equation*}
and the  Peierls  barrier  as
\begin{equation*}
h^{m}(x, y) = \liminf_{\tau \to +\infty} \big[A^{m}(\tau, x, y) + \alpha^m \tau\big]
\end{equation*}
where $\alpha^m$ is the Ma\~n\'e critical value  
\begin{equation*}
\alpha^m := \inf\{\alpha \in \R : \exists\; u \in C^1(\T^d) \mbox{ s.t. } \cH(x, Du(x),m) \leq \alpha\}.
\end{equation*}
\begin{proposition}\label{prp:Peierls}
For every $x \in \T^d$, the function $h^m_x : \T^d \to \R$, defined by
\[
h^m_x(y) := h^m(x, y)  \qquad y\in\T^d, 
\]
satisfies   the following equivalent conditions: 
\begin{enumerate}
\item $h^m_x$ is a viscosity solution to \eqref{eq:HJerg}.
\item For all $\tau \geq 0$ we have 
\begin{equation}\label{peierls}
h^m_x(y) + \alpha^m \tau = \inf_{\gamma(\tau) = y} \left\{h^m_x(\gamma(0)) + \int_{0}^{\tau} \cL(\gamma(s), \dot\gamma(s),m)\;ds \right\}.
\end{equation}
\item For any Lipschitz curve $\gamma:[a, b] \to \T^d$ we have 
\begin{equation*}\tag{{\it Dominated curve}}
h^m_x(\gamma(b)) - h^m_x(\gamma(a)) \leq  \int_{a}^{b} \cL(\gamma(s), \dot\gamma(s),m)\;ds + \alpha^m (b-a)
\end{equation*}
and, moreover, for every $x \in \T^d$ there exists a Lipschitz curve $\gamma_x: (-\infty, 0] \to \T^d$ such that $\gamma(0) = x$ and 
\begin{equation*}\tag{\it Calibrated curve}
h^m_x(\gamma_x(b)) - h^m_x(\gamma_x(a))= \int_{a}^{b} \cL(\gamma_x(s), \dot\gamma_x(s),m)\;ds + \alpha^m (b-a)
\end{equation*}
for any $a < b \leq 0$.
\end{enumerate}
\end{proposition}

\begin{definition}
We call projected Aubry set the nonempty compact subset of $\T^d$ defined by 
\begin{equation}\label{eq:aubry}
\mathcal{A}^m := \{x \in \T^d : h^m(x, x) = 0\}.
\end{equation} 
\end{definition}
We recall that the projected Aubry set is nonempty because the $\alpha$-limit set of the Euler flow, due to the compactness of the state space, is nonempty and is contained within $\mathcal{A}^m$. Furthermore, it is important to note that both the Peierls barrier and the Aubry set depend on the fixed measure $m$.\\
We give two properties, well known in this framework, that we will be exploited in the next section to prove the existence of solutions to MFG system. For the proofs of the following  results, we refer to \cite{Fa} and \cite{fr}.
\begin{proposition}\label{bounded_velocity}
For any $m \in \mathcal{P}(\T^d)$, for any $t \in [0,T]$ and for any $x   \in \T^d$ let $\gamma_{x}$ be a calibrated curve for $h^{m}(x, \cdot)$, that is, 
\begin{equation*}
h^{m}(x, \gamma_x(b)) - h^{m}(x, \gamma_x(a)) = \int_{a}^{b} \cL(\gamma_x(s), \dot\gamma_x(s), m)\;ds + \alpha^{m}(b-a)
\end{equation*}
for any $0 \leq a < b \leq T$. Then, we have 
\begin{equation*}
|\dot\gamma_x(t)| \leq \kappa(T). 
\end{equation*}
\end{proposition}
\begin{proposition}\label{ManeContinuity}
There exists $\bar R>0$ such that, for every $m_1, m_2 \in   \mathcal{P}(\T^d)$, there holds
\[
|\alpha^{m_1} - \alpha^{m_2}| \leq \omega_{\bar R}\left( \Wass(m_1, m_2)\right) 
\] 
where $\omega_R$ is the modulus of continuity introduced in \eqref{eq:cH_m}.
\end{proposition}
\proof
We recall that the Ma\~n\'e critical value can be written as
\begin{equation*}
\alpha^{m } = -\inf_{\varphi \in C^1(\T^d)} \sup_{x \in \T^d} \cH(x, D\varphi(x), m ).
\end{equation*}
By assumption~$({\bf H})$ (in particular, the coercivity of~$\cH$ w.r.t.~$p$), we deduce that there exists a positive constant $\bar R$ (independent of~$m$) such that
\begin{equation*}
\alpha^{m } = -\inf_{\substack{\varphi \in C^1(\T^d) \\ \|D\phi\|_\infty\leq \bar R}}\; \sup_{x \in \T^d}\; \cH(x, D\varphi(x), m ). 
\end{equation*}
By relation~\eqref{eq:cH_m} we easily deduce the statement.
\qed


\section{The quasi-stationary Mean Field Games system: existence of solutions}\label{quasistat_MFG}
In this section, we study the existence of solutions to the first order quasi-stationary MFG system 
\begin{equation}\label{eq:MFG_product}
\begin{cases}
\cH(x, Du(t ), m(t)) = \alpha^{m(t)}, & x \in \T^d,\; \forall\; t  \in [0,T],
\\
\partial_t m(t) - \mbox{div}\big(m(t)D_p\cH(x, Du(t ), m(t))\big)= 0, & (t, x) \in (0,T] \times \T^d,
\\
m(0) = m_0 & x \in \T^d,
\end{cases}
\end{equation}
by means of the weak KAM theory techniques introduced in Section \ref{Peierls}. 

\begin{remark}\em
Note that there are \emph{two} distinct time scales in the ergodic Hamilton-Jacobi equation from the previous system: the exogenous time   $t$, which is present due to the distribution $m(t)$ and, for a fixed $t \in [0, T]$, the intrinsic time $\tau \in \R$ associated with the Peierls barrier, see \eqref{peierls}. Additionally, the Peierls barrier and the Aubry set, being functions of $m(t)$, vary over time. This presents a significant challenge in studying \eqref{eq:MFG_product}, as the continuity properties of these elements with respect to perturbations are generally unknown.
\end{remark} 


\begin{theorem}\label{thm:1}
 Assume {\bf (H)} and  
\begin{itemize}
\item[{\bf (A)}] For any $m \in C([0,T]; \mathcal{P}(\T^d))$ there exists a unique $x_{m} \in \T^d$ such that 
\begin{equation*}
x_{m} \in \bigcap_{t \in [0,T]} \mathcal{A}^{m(t)} ;
\end{equation*}
\item[({\bf IC})] $m_0$ is a Borel probability measure  absolutely continuous  w.r.t. the Lebesgue measure and the density, still denoted by $m_0$, belongs to $L^\infty(\T^d)$.
\end{itemize}
Then, there exists a solution $(u,\alpha,  m)$ to \eqref{eq:MFG_product} such that 
\[
u(t, y) = h^{m(t)}(x_m, y), \quad \forall\; (t, x) \in [0,T] \times \T^d,
\]
where $x_m$ is the point given in {\bf (A)}, and $\alpha=\alpha^{m(t)}$ is the corresponding  Ma\~n\'e critical value. Moreover, the following properties hold.
\begin{itemize}
\item[i)] $\Wass(m(t), m(s)) \leq C_b|t-s|$ for every $t,s\in[0,T]$, where $C_b$ is a constant depending only on the assumptions (see, the proof of Lemma~\ref{FP_existence});
\item[ii)] for each $t\in[0,T]$, the measure $m(t)$ is absolutely continuous w.r.t. the Lebesgue measure and its density (that we still denote $m(t)$) belongs to $L^\infty(\T^d)$ with $\|m(t)\|_\infty\leq C$ where $C$ is a constant independent of~$t$;
\item[iii)] for any $t\in[0,T]$, there holds: $m(t)=X(t,\cdot)\# m_0$, where $X(t,\cdot)$ is the flux given by
\[
X(t,x)=x+\int_0^t D_p\cH\left(X(s,x), D_y h^{m(s)}(x_m, X(s,x)), m(s)\right)ds;
\]
\item[iv)] $u$ is Lipschitz continuous w.r.t. $x$, uniformly w.r.t. $t$;
\item[v)] $u$ is semiconcave w.r.t. $x$, uniformly w.r.t. $t$, and continuous w.r.t. $t$;
\item[vi)] $u(t, \cdot)$ is a weak KAM critical solution, namely it fulfills the properties of Proposition~\ref{prp:Peierls}.
\end{itemize}
\end{theorem}

Before passing to the proof of Theorem \ref{thm:1} we provide some examples of quasi-stationary MFG model which fit our assumptions {\bf (H)} and {\bf (A)}. 

\begin{example}\em
\fbox{\textit{(i)}} The basic example is the one of the mechanical Hamiltonian with a positive coupling cost, i.e.,
\begin{equation}\label{ex:H_quad}
\cH(x,p,m)=\dfrac{|p|^2}{2}-F(x, m),
\end{equation}
where $F$ is  continuous in $\T^d\times \cP(\T^d)$ and for any $m\in\cP(\T^d)$
\begin{equation}\label{ex:coupling}
F(0,m)=0\quad\text{and} \quad F(x,m)>0 \,\text{ for $x\neq 0$}.
\end{equation}
For example, given $k\in C(\T^d\times\T^d)$ with $k(0,y)\equiv 0$ and $k(x,y) >0$ in $\T^d\setminus\{0\}\times\T^d$, the function~ 
\[
F(x,m)=\int_{\T^d}k(x,y)m(dy)
\]
fulfills the previous properties. In this case, we have
\[\alpha^m=0 \quad \text{ and} \quad \mathcal{A}^m=\{0\} \qquad \text{ for any  $m\in\cP(\T^d)$} \]
(see \cite[Section 4.14]{Fa}), hence the assumption {\bf (A)} is satisfied. We can replace the quadratic Hamiltonian in \eqref{ex:H_quad} with any reversible Tonelli Hamiltonian $H(p)$, i.e. $H(p)=H(-p)$ for any $p\in\R^d$, such that $H(0)=0$ and $H(p)>0$ for $|p|\neq 0$.

\vspace{0.2cm}
\fbox{\textit{(ii)}} We can also consider a non-separable Hamiltonian of the type
\begin{equation*}
\cH(x,p,m)=F(m)(H(p)- V(x))
\end{equation*}
with $H$ a reversible Tonelli Hamiltonian,  $F(m)\ge \delta>0$  for any $m\in\cP(\T^d)$ and $V:\T^d\to \R$ a continuous function with a unique minimum point $x_0$. In this case,
\[
\alpha^m=-F(m)\displaystyle{\min_{\T^d}}(H(0)- V(x))=-F(m)(H(0)-V(x_0))
\]
and  $\mathcal{A}^m=\{x_0\}$ for any  $m\in\cP(\T^d)$. \qed
\end{example}

For the proof of Theorem \ref{thm:1} we need some preliminary lemmas. For these results, assumptions~$({\bf A})$ and $({\bf IC})$ are not needed.
\begin{lemma}\label{time_continuity}
Assume $({\bf H})$. Let $m \in C([0,T]; \mathcal{P}(\T^d))$ with $\bigcap_{t \in [0,T]} \mathcal{A}^{m(t)}$ not empty and fix any point $x_m \in \displaystyle{\bigcap_{t \in [0,T]}} \mathcal{A}^{m(t)}$. Then, the map 
\begin{equation*}
t \mapsto h^{m(t)}(x_m, y)
\end{equation*}
is continuous for any $y \in \T^d$. 
\end{lemma}
 
\proof 

Let $t \in [0,T]$ and let $\{t_n\}_{n \in \N} \subset [0,T]$ be such that $t_n \to t$ as $n \uparrow \infty$. Then, since the function 
\[
y \mapsto h^{m(t_n)}(x_m, y)
\]
is a global critical solution and $x_m$ belongs to the Aubry set, we have that there exists $\{\gamma_n\}_{n \in \N}$ such that $\gamma_n : [0,1] \to \T^d$, $\gamma_n(0)= x_m$ (where $x_m$ is as in the statement), $\gamma_n(1)=y$ and 
\begin{equation}\label{eq_1}
h^{m(t_n)}(x_m, y) - h^{m(t_n)}(x_m, x_m) = \int_{0}^{1} \cL(\gamma_n(s), \dot\gamma_n(s), m(t_n))\;ds + \alpha^{m(t_n)}. 
\end{equation}
Moreover, by the domination property we also have 
\begin{equation}\label{eq_2}
h^{m(t)}(x_m, y) - h^{m(t)}(x_m, x_m) \leq \int_{0}^{1} \cL(\gamma_n(s), \dot\gamma_n(s), m(t))\;ds + \alpha^{m(t)}. 
\end{equation}
Hence, recalling that $h^{m(s)}(x_m, x_m) = 0$ for any $s \in [0,T]$,  taking the difference of \eqref{eq_2} and \eqref{eq_1} we deduce
\begin{multline*}
h^{m(t)}(x_m, y) - h^{m(t_n)}(x_m, y) \\ \leq \int_{0}^{1} \left[\cL(\gamma_n(s), \dot\gamma_n(s), m(t))- \cL(\gamma_n(s), \dot\gamma_n(s), m(t_n))\right]\;ds + \alpha^{m(t)} - \alpha^{m(t_n)}. 
\end{multline*}
From \Cref{bounded_velocity}, we obtain $\|\dot\gamma_n\|_\infty\leq \kappa(T)$; hence, by relation~\eqref{eq:cH_m} and \Cref{ManeContinuity} we accomplish the proof.
\qed

\begin{lemma}\label{lemma:barrier}
Assume $({\bf H})$. Then, the following properties hold.
\begin{itemize}
\item[($i$)] The map $y \mapsto h^{m(t)}(x, y)$ is semiconcave with a linear modulus on $\T^d$, uniformly w.r.t. $t \in [0,T]$ and $x \in \T^d$.
\item[($ii$)] Given 
\[
x_m \in \bigcap_{t \in [0,T]} \mathcal{A}^{m(t)}
\]
the map $t \mapsto D_y h^{m(t)}(x_m, y)$ is measurable for any $y \in \T^d$.
\end{itemize}
\end{lemma}
\proof

We recall that given $t \in [0,T]$ and $x \in \T^d$ the function $y \mapsto h^{m(t)}(x, y)$ is a viscosity solution of the critical equation 
\begin{equation*}
\cH(y, D_y h^{m(t)}(x, y), m(t)) = \alpha^{m(t)}, \quad y \in \T^d;
\end{equation*}
the semiconcavity estimates, uniform w.r.t. time, follows by classical arguments (see \cite{Fa}). Statement ($ii$) is a standard consequence of \Cref{time_continuity}. \qed

\begin{lemma}\label{FP_existence}
Let $({\bf H})$ be in force and assume that  $m\in C([0,T]; \mathcal{P}(\T^d))$ is a solution in the sense of distributions of 
\begin{equation}\label{transport}
\begin{cases}
\partial_t m(t) - \mbox{div}\big(m(t) D_{p}H(y, D_y h^{m(t)}(x_m, y))F(y, m(t)) \big) = 0, & (t, y) \in [0,T] \times \T^d
\\
m(0) = m_0, & y \in \T^d.
\end{cases}
\end{equation}
Then,  the map $t \mapsto m(t)$ is Lipschitz continuous on~$[0,T]$ w.r.t. the $\Wass$ distance with Lipschitz constant $C_b$ independent of~$m$.
\end{lemma}
\proof 
Let $m\in C([0,T]; \mathcal{P}(\T^d))$ be a solution in the sense of distributions to \eqref{transport}. By definition and by a standard density argument for $\phi$ near times~$t$ and~$s$ (see \cite[Lemma 8.1.2]{ags}), for any $\varphi\in C^1(\T^d)$ and any $0 \leq s < t \leq T$, we have that 
\begin{multline*}
\int_{\T^d} \varphi(y) \big( m(t, dy) - m(s, dy) \big) \\ + \iint_{(s, t) \times \T^d} D\varphi(y) D_p\cH(y, D_y h^{m(\sigma)}(x_m, y), m(\sigma))\; m(\sigma, dy)d\sigma =0.
\end{multline*}
Then, taking $\psi \in C(\T^d)$ such that $|\psi(x) - \psi(y)| \leq |x-y|$ and defining $\psi_\eps = \psi \star \xi_{\eps}$, where $\xi_{\eps}$ is a smooth mollifier, the previous equality gives the following estimate
\begin{multline*}
\int_{\T^d} \psi_{\eps}(y) \big( m(t, dy) - m(s, dy) \big) \\ \leq \iint_{(s, t) \times \T^d} |D\psi_{\eps}(y)||D_p\cH(y, D_y h^{m(\sigma)}(x_m, y), m(\sigma))|\; m(\sigma, dy)d\sigma \\ \leq \|D\psi\|_{\infty}c_H(1 + \sup_{\sigma\in[s,t]}\|D_y h^{m(\sigma)}(x_m, \cdot)\|_{\infty}) |t-s|. 
\end{multline*}
By taking the limit as $\eps \downarrow 0$ in the previous inequality and  recalling that $ \|D\psi\|_{\infty} \leq 1$, the arbitrariness of~$\psi$ yields
\begin{equation}\label{Lip}
\Wass(m(t), m(s)) \leq c_H\left(1 + \sup_{\sigma\in[s,t]}\|D_y h^{m(\sigma)}(x_m, \cdot)\|_{\infty}\right) |t-s|. 
\end{equation}
In order to get the uniform bound for the Lipschitz constant, 
it is enough to obtain that the functions $y\mapsto h^{m(t)}(x_{m}, \cdot)$ are Lipschitz continuous with a Lipschitz constant independent of $m(t)$. Indeed, by standard arguments, for any $m$ and any $t$, there holds
\begin{equation*}|\alpha^{m(t)}|\leq \sup_{(y,m)\in\T^d\times\cP(\T^d)}|H(y,0,m)|=:M_\alpha;
\end{equation*}
hence, by \Cref{prp:Peierls} and {\bf (H)} , we infer 
\[
\|D_y h^{m(t)}(x_m, \cdot)\|_{\infty} \leq \left(\frac {M_\alpha+ c_H }{c_H}\right)^{1/2}.
\]
In conclusion, choosing
\[
C_b:= c_H^{1/2}\left( c_H^{1/2} +(M_\alpha+ c_H )^{1/2}\right)
\]
we accomplish the proof.
\qed

\begin{proof}[Proof of Theorem \ref{thm:1}]
In order to show the existence of solutions to \eqref{transport} we use the Schauder fixed-point theorem. To do so, we consider the set
\begin{equation*}
\mathcal{C}=\left\{m \in C([0,T]; \mathcal{P}(\T^d)): \sup_{t \not=s} \frac{\Wass(m(t), m(s))}{|t-s|} \leq C_b\right\},
\end{equation*}
where $C_b$ is the constant introduced in Proposition~\ref{FP_existence} 
and define the map 
\[
S: \mathcal{C} \to \mathcal{C}
\]
as follows: given $m \in \mathcal{C}$ we fix 
the Peierls barrier $(t, y) \mapsto h^{m(t)}(x_m, y)$, where $x_m$ is the point defined in {\bf (A)}, and we define
\[
S(m) := \mu
\]
where $\mu$ is a solution to
\begin{equation}\label{eq:pt_fisso}
\begin{cases}
\partial_t \mu(t) - \mbox{div}\big(\mu(t) D_{p}\cH(y, D_y h^{m(t)}(x_m, y), m(t)) \big) = 0 & (t, y) \in [0,T] \times \T^d
\\
m(0) = m_0 & x \in \T^d.
\end{cases}
\end{equation}
Note that $\mu$ exists and it is unique by \cite[Section 5]{ac}; indeed, by \cite[Theorem 2.3.1-(i) and theorem A.6.5]{cs} and Lemma~\ref{lemma:barrier} the function $D_y[ D_{p}\cH(\cdot, D_y h^{m(t)}(x_m, \cdot), m(t))]$ is absolutely continuous w.r.t. the Lebesgue measure for a.e. $t\in[0,T]$. Hence, \cite[Section 5]{ac} ensures the existence of a Lagrangian flow (see \cite{ac} for its definition and main properties) associated to   \eqref{eq:pt_fisso}; in particular the function $\mu(t)=X(t,\cdot)\# m_0$, where $X(t,\cdot)$ is the flux given by
\[
X(t,x)=x+\int_0^t D_p\cH\left(X(s,x), D_y h^{m(s)}(x_m, X(s,x)),m(s)\right)ds,
\]
is a solution to~\eqref{eq:pt_fisso}; moreover, for each $t\in[0,T]$, $\mu(t)$ is absolutely continuous w.r.t. the Lebesgue measure and its density belongs to $L^\infty(\T^d)$ with $\|\mu(t)\|_\infty\leq C$ where $C$ is a constant independent of~$t$.
Furthermore, by the Lipschitz estimates established in Proposition~\ref{FP_existence}, $\mu$ belongs to $\mathcal{C}$; hence the map~$S$ is well-posed. Then, since $\mathcal{C}$ is convex and compact, in order to apply the Schauder fixed-point theorem it suffices to show that $S$ is continuous.

Let $\{m_n\}_{n \in \N} \in \mathcal{C}$ be such that $m_n \rightharpoonup m$ in $\mathcal{C}$, let $\{x_n\} \subset \T^d$ be the sequence of points which satisfies {\bf (A)} for $m_n$, i.e for each $n\in\N$ there exists a unique point~$x_n$ in $\bigcap_{t \in [0,T]} \mathcal{A}^{m_n(t)}$, and let
\[
\mu_n = S(m_n).
\]
By compactness of $\T^d$, possibly passing to a subsequence, the sequence $\{x_n\}_{n\in N}$ converges to some point~$\overline{x}$. We claim that 
\[
\overline{x} \in \bigcap_{t \in [0,T]} \mathcal{A}^{m(t)}. 
\]
To do so, it is enough to prove that, for each $t\in[0,T]$, the point~$\bar x$ belongs to $\mathcal{A}^{m(t)}$. Considering 
\begin{equation*}
h^{m_n(t)}(x_n, x_n) = 0
\end{equation*}
by definition there exist $\{\tau_n\}_{n \in \N}$ and $\{\gamma_n\}_{n \in \N}$ such that $\tau_n \uparrow \infty$,  $\gamma_n: [0, \tau_n] \to \T^d$, $\gamma_{n}(0) = \gamma_n(\tau_n)= x_n$ and 
\begin{equation*}
\int_{0}^{\tau_n} \cL(\gamma_n(s), \dot\gamma_n(s), m_n(t))\;ds + \alpha^{m(t)}\tau_n \leq \frac{1}{n}. 
\end{equation*}
By \Cref{bounded_velocity} and Ascoli-Arzela theorem, there exists $\overline\gamma$ such that $\gamma_n$ uniformly converges to $\overline\gamma$ and $\dot\gamma_n$ weakly converges to $\dot{\overline\gamma}$ in $L^2(0,\infty; \T^d)$, on every compact subset of $[0, \infty)$ respectively. Set
\[
d_n = |x_n - \overline{x}|
\]
and define the curve 
\begin{equation*}
\widetilde\gamma_n(s) = 
\begin{cases}
\gamma_n^1 = \mbox{segment } \overline{x} \to x_n, & s \in [-d_n, 0]
\\
\gamma_n(s), & s \in (0, \tau_n]
\\
\gamma_n^2 = \mbox{segment } x_n \to \overline{x}, & s \in (\tau_n, \tau_n + d_n].
\end{cases}
\end{equation*}
Up to a reparametrization, we can assume that $|\dot\gamma_n^i(s)| \leq 1$ (for $i=1, 2$) which yields
\begin{equation*}
\int_{-d_n}^{0} \cL(\gamma_n^1(s), \dot\gamma_n^1(s), m_n(t))\;ds \leq d_n \|\cL\|_{\infty, \T^d \times \overline{B}_1\times\cP(\T^d)} 
\end{equation*}
and the same estimate holds for $\gamma^2_n$ in $[\tau_n, \tau_n+d_n]$. Hence, by lower-semicontinuity of the action functional we have 
\begin{multline*}
h^{\overline{m}(\tau)}(\overline{x}, \overline{x}) = \liminf_{\tau \to \infty} \left\{ \inf_{\gamma(0)=\gamma(\tau) = \overline{x}}\int_{0}^{\tau}\cL(\gamma(s), \dot\gamma(s),  {m}(t))\;ds + \alpha^{m(t)}\tau \right\}
\\
\leq \liminf_{n \to \infty} \left\{ \int_{-d_n}^{\tau_n+d_n} \cL(\widetilde\gamma_n(s), \dot{\widetilde{\gamma}_n}(s), m_n(t))\;ds + \alpha^{m(t)} (\tau_n + 2d_n) \right\}
\\
\leq \liminf_{n \to \infty} \Big(\int_{-d_n}^{0} \cL(\gamma_n^1(s), \dot\gamma_n^1(s), m_n(t))\;ds + \int_{0}^{\tau_n} \cL(\gamma_n(s), \dot\gamma_n(s), m_n(t))\;ds + \alpha^{m(t)}\tau_n \\ + \int_{\tau_n}^{\tau_n+d_n} \cL(\gamma_n^2(s), \dot\gamma_n^2(s), m_n(t))\;ds + 2\alpha^{m(t)}d_n \Big) 
\\
\leq \lim_{n \to \infty} \left(2d_n\sup_{t \in [0,T]}|\alpha^{m(t)}| \|\cL\|_{\infty, \T^d \times \overline{B}_1\times\cP(\T^d)} + \frac{1}{n} \right) = 0
\end{multline*}
which proves that $\overline{x} \in \mathcal{A}_{\overline{m}}(t)$; by the arbitrariness of $t \in [0,T]$ and by assumption {\bf (A)} 
we infer
\begin{equation*}
\{\overline{x} \}=\bigcap_{t \in [0,T]} \mathcal{A}_{\overline{m}}(t).
\end{equation*}
Now, by continuity of the maps $x \mapsto h^{m(t)}(x, y)$ (see, e.g., \cite{Fa}) we have that for any $(t, y) \in [0,T] \times \T^d$ the limit as $n \to \infty$ of
\begin{equation*}
h^{m_n(t)}(x_n, y) \to h^{m(t)}(\overline{x}, y).
\end{equation*}
Moreover, still from the continuity of the map $x \mapsto h^{m(t)}(x, y)$ we deduce that $x \mapsto D_y h^{m(t)}(x, y)$ is measurable. Thus, combining such a property with the uniform semiconcavity of the map $y \mapsto h^{m(t)}(x, y)$ w.r.t. $t$, we also obtain 
\begin{equation*}
D_y h^{m_m(t)}(x_n, y) \to D_y h^{m(t)}(\overline{x}, y), \quad \mbox{on a.e. } (t, y) \in [0,T] \times \T^d
\end{equation*}
invoking \cite{cs}.

Therefore, by compactness of $\mathcal{C}$ we get $\mu_n \rightharpoonup \overline\mu$, up to a subsequence for some $\overline\mu \in \mathcal{C}$. Passing to the limit into the equation associated with $m_n$ we deduce that $\overline\mu$ solve the equation associated with $h^{\overline\mu(t)}(\overline{x}, y)$, i.e., 
\begin{equation*}
\begin{cases}
\partial_t \overline\mu(t) - \mbox{div}\big(\overline\mu(t) D_{p}\cH(y, D_y h^{\overline{\mu}(t)}(\overline{x}, y), \overline\mu(t)) \big) = 0, & (t, y) \in [0,T] \times \T^d
\\
\overline\mu(0) = m_0, & y \in \T^d.
\end{cases}
\end{equation*}
Thus, recalling that 
\begin{equation*}
\alpha^{m(t)} + \cH(y, D_y h^{\overline{\mu}(t)}(\overline{x}, y), \overline\mu(t)) = 0, \quad y \in \T^d, \; \forall\; t \in [0,T],
\end{equation*}
we obtain the existence of a solution.\\
In conclusion, point $(i)$ is an immediate consequence of our construction of $m$ while points $(ii)$ and $(iii)$ are due to \cite{ac}. Moreover, points $(iv)$-$(vi)$ are due respectively to Proposition~\ref{FP_existence} and its proof, Lemma~\ref{lemma:barrier} and \cite[Proposition 2.2]{fr}.
\end{proof}

\small{
}	
\end{document}